\documentclass[10pt,twoside]{article}
\usepackage[all]{xy}
\usepackage{amsmath,amssymb,amsthm,bbm,mathrsfs,url,xspace}
\usepackage{verbatim} 
\newtheorem{Pro}{Proposition}
\newtheorem{Lem}[Pro]{Lemma}
\newtheorem{theorem}[Pro]{Theorem}
\newtheorem{Cor}[Pro]{Corollary}
\newtheorem*{Property}{Property}
\newtheorem*{Pb}{Problem}
\theoremstyle{definition}
\newtheorem{Def}[Pro]{Definition}
\theoremstyle{remark}
\newtheorem{Conv}[Pro]{Convention}
\newtheorem{Exm}[Pro]{Example}
\newtheorem{Rem}[Pro]{Remark}
\newtheorem{Not}[Pro]{Notations}
\newtheorem*{commentVF}{VF's comment} 

\def\SL{{\mathrm{SL}}}
\def\St{St_r}
\let\si\sigma

\DeclareMathOperator{\Spec}{Spec}

\def\sl{{\mathfrak {sl}}}
\def\con{{\mathfrak c}}
\def\p{{\mathfrak p}}
\def\q{{\mathfrak q}}
\def\rad{{\mathfrak r}}
\DeclareMathOperator{\Frac}{Frac}
\DeclareMathOperator{\Proj}{Proj}

\def\F{{\mathbb F}}
\def\N{{\mathbb N}}
\def\Z{{\mathbb Z}}
\def\Q{{\mathbb Q}}
\def\P{{\mathbb P}}
\def\G{{\mathbb G}}
\def\k{{\mathbf k}}

\def\L{{\mathcal L}}
\def\I{{\mathfrak I}}
\def\M{{\mathfrak M}}
\DeclareMathOperator{\hgt}{ht}

\let\x\times
\let\ot\otimes

\DeclareMathOperator{\hull}{hull}
\def\even{{\mathrm{even}}}
\def\tors{{\mathrm{tors}}}
\def\pos{{\mathrm{pos}}}
\DeclareMathOperator{\gr}{gr}
\def \ind{{\mathrm{ind}}}

\def \ker{{\mathrm{Ker}}}
\def \hom{{\mathrm{Hom}}}
\def \ext{{\mathrm{Ext}}}
\def \h{{\mathrm{H}}}
\def \id{{\mathrm{id}}}
\def \triv{{{\mathrm{triv}}}}
\def\YEAR{\year}\newcount\VOL\VOL=\YEAR\advance\VOL by-1995
\def\XVOL{$\cdot$ Extra Volume Suslin}
\def\firstpage{171}\def\lastpage{195}
\def\received{November 27, 2008}\def\revised{September 29, 2009}
\def\communicated{}

\makeatletter
\def\magnification{\afterassignment\m@g\count@}
\def\m@g{\mag=\count@\hsize6.5truein\vsize8.9truein\dimen\footins8truein}
\makeatother

\font\eightrm=cmr8
\font\caps=cmcsc10                    
\font\Caps=cmcsc10 scaled \magstep1   
\font\scaps=cmcsc8

\makeatletter
\@specialpagefalse\headheight=8.5pt
\def\DocMath{{\def\th{\thinspace}\scaps Documenta Math.}}
\renewcommand{\@oddfoot}{\hfill\scaps Documenta Mathematica 
    \XVOL\  (\number\YEAR) \number\firstpage--\lastpage\hfill}
\renewcommand{\@evenfoot}{\ifnum\thepage>\lastpage\hfill\scaps
    Documenta Mathematica \XVOL\  (\number\YEAR)\hfill\else\@oddfoot\fi}%
\renewcommand{\@evenhead}{%
    \ifnum\thepage>\lastpage\rlap{\thepage}\hfill%
    \else\rlap{\thepage}\slshape\leftmark\hfill{\caps\SAuthor}\hfill\fi}%
\renewcommand{\@oddhead}{%
    \ifnum\thepage=\firstpage{\DocMath\hfill\llap{\thepage}}%
    \else{\slshape\rightmark}\hfill{\caps\STitle}\hfill\llap{\thepage}\fi}%
\makeatother

\def\TSkip{\bigskip}
\newbox\TheTitle{\obeylines\gdef\GetTitle #1
\ShortTitle  #2
\SubTitle    #3
\Author      #4
\ShortAuthor #5
\EndTitle
{\setbox\TheTitle=\vbox{\baselineskip=20pt\let\par=\cr\obeylines%
\halign{\centerline{\Caps##}\cr\noalign{\medskip}\cr#1\cr}}%
	\copy\TheTitle\TSkip\TSkip%
\def\next{#2}\ifx\next\empty\gdef\STitle{#1}\else\gdef\STitle{#2}\fi%
\def\next{#3}\ifx\next\empty%
    \else\setbox\TheTitle=\vbox{\baselineskip=20pt\let\par=\cr\obeylines%
    \halign{\centerline{\caps##} #3\cr}}\copy\TheTitle\TSkip\TSkip\fi%
\centerline{\caps #4}\TSkip\TSkip%
\def\next{#5}\ifx\next\empty\gdef\SAuthor{#4}\else\gdef\SAuthor{#5}\fi%
\ifx\received\empty\relax
    \else\centerline{\eightrm Received: \received}\fi%
\ifx\revised\empty\TSkip%
    \else\centerline{\eightrm Revised: \revised}\TSkip\fi%
\ifx\communicated\empty\relax
    \else\centerline{\eightrm Communicated by \communicated}\fi\TSkip\TSkip%
\catcode'015=5}}\def\Title{\obeylines\GetTitle}
\def\Abstract{\begingroup\narrower
    \parskip=\medskipamount\parindent=0pt{\caps Abstract. }}
\def\EndAbstract{\par\endgroup\TSkip}

\long\def\MSC#1\EndMSC{\def\arg{#1}\ifx\arg\empty\relax\else
     {\par\narrower\noindent%
     2010 Mathematics Subject Classification: #1\par}\fi}

\long\def\KEY#1\EndKEY{\def\arg{#1}\ifx\arg\empty\relax\else
	{\par\narrower\noindent Keywords and Phrases: #1\par}\fi\TSkip}

\newbox\TheAdd\def\Addresses{\vfill\copy\TheAdd\vfill
    \ifodd\number\lastpage\vfill\eject\phantom{.}\vfill\eject\fi}
{\obeylines\gdef\GetAddress #1
\Address #2 
\Address #3
\Address #4
\EndAddress
{\def\xs{5.3truecm}\parindent=0pt
\setbox0=\vtop{{\obeylines\hsize=\xs#1\par}}\def\next{#2}
\ifx\next\empty 
     \setbox\TheAdd=\hbox to\hsize{\hfill\copy0\hfill}
\else\setbox1=\vtop{{\obeylines\hsize=\xs#2\par}}\def\next{#3}
\ifx\next\empty 
     \setbox\TheAdd=\hbox to\hsize{\hfill\copy0\hfill\copy1\hfill}
\else\setbox2=\vtop{{\obeylines\hsize=\xs#3\par}}\def\next{#4}
\ifx\next\empty\ 
     \setbox\TheAdd=\vtop{\hbox to\hsize{\hfill\copy0\hfill\copy1\hfill}
                \vskip20pt\hbox to\hsize{\hfill\copy2\hfill}}
\else\setbox3=\vtop{{\obeylines\hsize=\xs#4\par}}
     \setbox\TheAdd=\vtop{\hbox to\hsize{\hfill\copy0\hfill\copy1\hfill}
	        \vskip20pt\hbox to\hsize{\hfill\copy2\hfill\copy3\hfill}}
\fi\fi\fi\catcode'015=5}}\gdef\Address{\obeylines\GetAddress}

\begin{document}
\Title
Power Reductivity over an Arbitrary Base
\ShortTitle 
\SubTitle
\`A Andr\'e Sousline pour son soixanti\`eme anniversaire
Voor Andree Soeslin op zijn zestigste verjaardag
\Author 
Vincent Franjou$^1$ and Wilberd van der Kallen$^2$
\ShortAuthor 
V. Franjou and W. van der Kallen
\EndTitle
\Abstract 
Our starting point is Mumford's conjecture, on representations of Chevalley groups over fields, as it is phrased in the preface of 
\emph{Geometric Invariant Theory}.  After  extending the conjecture appropriately, we show that it holds over an arbitrary commutative base ring. 
We thus obtain the first fundamental theorem of invariant theory (often referred to as Hilbert's fourteenth problem) 
over an arbitrary Noetherian ring.
We also prove results on the Grosshans graded deformation of an algebra in the same generality. 
We end with tentative finiteness results for rational cohomology over the integers.
\EndAbstract
\MSC 20G35; 20G05; 20G10
\EndMSC
\KEY Chevalley group; Hilbert's 14th; 
cohomology; Geometric reductivity.
\EndKEY
\Address 
Vincent Franjou
Math\'ematiques
UFR Sciences \& Techniques
BP 92208
F-44322 Nantes Cedex 3
Vincent.Lastname@univ-nantes.fr
\Address
Wilberd van der Kallen
Mathematisch Instituut
Universiteit Utrecht
P.O. Box 80.010
NL-3508 TA Utrecht
Initial.vanderKallen@uu.nl
\Address
\Address
\EndAddress

\footnotetext[1]{LMJL - Laboratoire de Math\'ematiques Jean Leray, CNRS/Universit\'e de
Nantes.
The author acknowledges the hospitality and support of CRM Barcelona
during the tuning of the paper.}
\footnotetext[2]{The project started in Nantes, the author being the first visitor in the MATPYL program of CNRS' FR 2962 ``Math\'ematiques des Pays de Loire''. }




\section{Introduction}
The following statement may seem quite well known:
\begin{theorem}\label{Bob:theorem}
Let $\k$ be a Dedekind ring and let $G$ be a Chevalley group scheme over $\k$.
Let $A$ be a finitely generated commutative $\k$-algebra on which $G$ acts rationally through algebra automorphisms.
The subring of invariants $A^G$ is then a finitely generated $\k$-algebra.
\end{theorem}
Indeed, R. Thomason proved \cite[Theorem 3.8]{T} the statement for any Noetherian Nagata ring $\k$.
Thomason's paper deals with quite a different theme, that is the existence of equivariant resolutions by free modules.
Thomason proves that equivariant sheaves can be resolved by equivariant vector bundles.
He thus solves a conjecture of Seshadri \cite[question 2 p.268]{Se}.
The affirmative answer to Seshadri's question is explained to yield Theorem \ref{Bob:theorem} in the same 
paper \cite[Theorem 2 p.263]{Se}.
The finesse only illustrates that the definition of geometric reductivity in \cite{Se} does not suit well an arbitrary base.
Indeed, Seshadri does not follow the formulation in Mumford's book's introduction \cite[Preface]{M}, 
and uses polynomials instead \cite[Theorem 1 p.244]{Se}. 
This use of a dual in the formulation seems to be why one
requires Thomason's result \cite[Corollary 3.7]{T}.
One can rather go back to the original formulation in terms of symmetric powers as illustrated by the following:
\begin{Def}\label{power reductive:Def}
Let $\k$ be a ring and let $G$ be an algebraic group over $\k$.
The group $G$ is \emph{power-reductive} over $\k$ if the following holds. 
\begin{Property}[Power reductivity]
Let $L$ be a cyclic $\k$-module with trivial $G$-action. 
Let $M$ be a rational $G$-module, and let $\varphi$ be a $G$-module map from $M$ onto $L$.
Then there is a positive integer $d$ such that the $d$-th symmetric power of $\varphi$ induces a surjection:
\[
(S^d M)^G\to S^d L .
\]
\end{Property}
\end{Def}
We show in Section \ref{Mumford:section} that power-reductivity holds for Chevalley group schemes $G$, without assumption on the commutative ring $\k$. 
Note that this version of reductivity is exactly what is needed in Nagata's treatment of finite generation of invariants. We thus obtain:
\begin{theorem}\label{FFT:theorem}
Let $\k$ be a Noetherian ring and let $G$ be a Chevalley group scheme over $\k$.
Let $A$ be a finitely generated commutative $\k$-algebra on which $G$ acts rationally through algebra automorphisms.
The subring of invariants $A^G$ is then a finitely generated $\k$-algebra.
\end{theorem}
There is a long history of cohomological finite generation statements as well,  
where the algebra of invariants $A^G=\h^0(G,A)$ is replaced by the whole algebra $\h^*(G,A)$ of the derived functors of invariants. 
Over fields, Friedlander and Suslin's solution in the case of finite group schemes \cite{FS} lead to the conjecture in \cite{vdK coh}, now a theorem of Touz\'e \cite{TvdK}. 
In Section \ref{Grosshans:section}, we generalize to an arbitrary (Noetherian) base Grosshans' results on his filtration \cite{G}. 
These are basic tools for obtaining finite generation statements on cohomology. 
In Section \ref{finiteness:section}, we apply our results in an exploration of the case when the base ring is $\Z$.
Section \ref{generalities:section} presents results of use in Section \ref{Grosshans:section} and Section \ref{Mumford:section}. Our results support the hope that Touz\'e's theorem extends to an arbitrary base.
\begin{comment}
Section 5 mentioned
\end{comment}
\begin{commentVF}
Now Section 6
\end{commentVF}

\section{Power reductivity and Hilbert's 14th}\label{power reductivity:section}
\subsection{Power surjectivity}
To deal with the strong form of integrality we encounter, we find it convenient to make the following definition.
\begin{Def}
A morphism of $\k$-algebras: $\phi:S\to R$ is \emph{power-surjective}
if every element of $R$ has a power in the image of $\phi$.
It is \emph{universally power-surjective}
if for every $\k$-algebra $A$, the morphism of $\k$-algebras $A\ot\phi$ is power-surjective, that is: 
for every $\k$-algebra $A$, for every $x$ in $A\ot R$, there is a positive integer $n$ so that $x^n$ lies in the image of $A\ot\phi$.
\end{Def}
If $\k$ contains a field, one does not need arbitrary positive exponents $n$, but only powers of the characteristic exponent of $\k$ (compare \cite[Lemma 2.1.4, Exercise 2.1.5]{Sp} or Proposition \ref{co t:Prop} below). Thus if $\k$ is a field of characteristic zero, any universally power-surjective morphism of $\k$-algebras is surjective.
%

\begin{comment}
One also meets this theme with Theorem \ref{hull and powers:theorem}.
But what does one do with it? Well, one needs it when the target is $p$-torsion. We will be using Proposition \ref{co t:Prop}.
\end{comment}
\begin{commentVF}
I kept the rest iversal, until we find another good use for universal.
\end{commentVF}
\subsection{Consequences}
We start by listing consequences of power reductivity, as defined in the introduction (Definition \ref{power reductive:Def}).
\begin{Conv}
An algebraic group over our commutative ring $\k$ is always assumed to be a flat affine group scheme over $\k$.
Flatness is essential, as we tacitly use throughout that the functor of taking invariants is left exact. 
\end{Conv}

\begin{Pro}[Lifting of invariants]\label{power:Pro}
Let $\k$ be a ring and let $G$ be a power-reductive algebraic group over $\k$.
Let $A$ be a finitely generated commutative $\k$-algebra on which $G$ acts rationally through algebra automorphisms. If $J$ is an invariant ideal in $A$, the map induced by reducing mod $J$:
\[
A^G\to(A/J)^G
\]
is power-surjective.
\end{Pro}
For an example over $\Z$, see \ref{power lift}.
\begin{commentVF}
But how much was known before us, i.e. what is the relation between geometric reductivity and power reductivity over $\Z$, say?
\end{commentVF}
\begin{comment}
No idea. Ah, see later comment: Noetherian property for invariants in modules implies power reductivity.
In fact property (Int) implies power reductivity.
\end{comment}
\begin{Rem}\label{pow inv:rem}
Let $G$ be power reductive and let $\phi:A\to B$ be a power-surjective $G$-map of $\k$-algebras.
One easily shows that $A^G\to B^G$ is power-surjective.
\end{Rem}
\begin{comment}
You like giving exercises? Is this the wrong place?
\end{comment}
\begin{commentVF}
I just moved it.
\end{commentVF}
\begin{theorem}[Hilbert's fourteenth problem]\label{fg:theorem}
Let $\k$ be a Noetherian ring and let $G$ be an algebraic group over $\k$.
Let $A$ be a finitely generated commutative $\k$-algebra on which $G$ acts rationally through algebra automorphisms. If $G$ is power-reductive, 
then the subring of invariants $A^G$ is a finitely generated $\k$-algebra.
\end{theorem}
\begin{comment}
Without flatness of $G$ everything fails: Taking invariants is probably not left exact. 
So the invariants in a submodule are not what you think. See example below.
Should we make a convention that algebraic group over $\k$ means flat affine group scheme 
of finite type over $\k$?
\end{comment}
\begin{proof}
We apply \cite[p. 23--26]{Sp}. 
\begin{comment}
One also needs to do the exercise 2.4.12 on page 26 !
Here I find it helpful to say:
If $f\in S$ has degree at most $d$, one has a unique homogeneous lift $\tilde f\in \tilde S_d$ with
$x^d\tilde f(v,x)=x^d f(x^{-1}v)$ formally. (Think of $\tilde S$ as a polynomial ring $S[t]$. Then the
lift is for the substitution $t\mapsto 1$.) 
As $\tilde S_d\to S_{\leq d}$ is a split map of $G$-modules, invariants lift. And they still lift when going modulo the ideal, as the map still splits.

Note also that one needs in the proof of Springer \cite[2.4.9]{Sp} that $G$ is flat over the base.
(Where one concludes that $aA\cap A^G=aA^G$. This needs that the fixed point functor is left exact.
Look at $0\to A\xrightarrow{\times a} A\to A/aA\to 0$.) 

In fact such a thing fails in the following non flat example. 
Take the group functor $G$ which associates to a ring $R$ the additive subgroup of two-torsion elements. So $\Z[G]=\Z[T]/2T\Z[T]$. Let $t$ be the class of $T$.
On $M=\Z^2$ we let $G$ act by the matrix $u(t):=\begin{pmatrix}
1&t\\0&1
\end{pmatrix}$. (Convert into comodule language.)
The invariants are the vectors whose second coordinate is a multiple of two. So $2M\cap M^G\neq 2(M^G)$.
\end{comment}
It shows that Theorem \ref{fg:theorem} relies entirely on the conclusion of Proposition \ref{power:Pro}, which is equivalent to the statement \cite[Lemma 2.4.7 p. 23]{Sp}
that, for a surjective $G$-map $\phi:A\to B$ of $\k$-algebras, the induced map on invariants $A^G\to B^G$ is power-surjective.
To prove that power reductivity implies this, consider an invariant $b$ in $B$, take for $L$ the cyclic module $\k.b$ and for $M$ any submodule of $A$ such that $\phi(M)=L$.
We conclude with a commuting diagram:
\[\xymatrix{
\ar[r]\ar[d]^{S^d\phi}(S^dM)^G&\ar[r]\ar[d] (S^dA)^G&\ar[d]^{\phi^{G}} A^G\ \\
\ar[r]S^dL&\ar[r] (S^dB)^G&B^G.
}\]
\end{proof}
\begin{theorem}[Hilbert's fourteenth for modules]\label{module fg:theorem}
Let $\k$ be a Noetherian ring and let $G$ be a power-reductive algebraic group over $\k$.
Let $A$ be a finitely generated commutative $\k$-algebra on which $G$ acts rationally, 
and let $M$ be a Noetherian $A$-module, with an equivariant structure map $A\ot M\to M$. 
If $G$ is power-reductive, then the module of invariants $M^G$ is Noetherian over $A^G$.
\end{theorem}
\begin{proof}
As in \cite[2.2]{vdK reductive}, consider either the symmetric algebra of $M$ on $A$, 
or the `semi-direct product ring' $A\ltimes M$ as in Proposition \ref{Hm is noetherian:Pro}, whose underlying $G$-module is $A\oplus M$, with product given by
$(a_1,m_1)(a_2,m_2)=(a_1a_2,a_1m_2+a_2m_1)$.
\end{proof}
\subsection{Examples}
\subsubsection{}Let $\k=\Z$.
Consider the group $\SL_{2}$ acting on $2\times 2$ matrices 
$\begin{pmatrix}
a&b\\c&d
\end{pmatrix}$
by conjugation. Let $L$ be the line of homotheties in $M:=M_{2}(\Z)$. 
Write  $V^\#$ to indicate the dual module $\hom_\Z(V,\Z)$ of a $\Z$-module $V$.
The restriction: $M^\#\to L^\#$ extends to
\[
\Z[M]=\Z[a,b,c,d]\to \Z[\lambda]=\Z[L].
\] 
Taking $\SL_{2}$-invariants:
\[
\Z[a,b,c,d]^{\SL_{2}}=\Z[t,D] \to \Z[\lambda],
\]
the trace $t=a+d$ is sent to $2\lambda$, so $\lambda$ does not lift to an invariant in $M^\#$.
The determinant $D=ad-bc$ is sent to $\lambda^2$ however, illustrating power reductivity of $\SL_{2}$.

\subsubsection{}\label{power lift}
Similarly, the adjoint action of $\SL_{2}$ on $\sl_2$  is such
that 
$u(a):=\begin{pmatrix}
1&a\\0&1
\end{pmatrix}$
sends $X,H,Y\in\sl_2$ respectively to $X+aH-a^2Y,H-2aY,Y$. This action extends to the symmetric algebra $S^*(\sl_2)$, which
is a polynomial ring in variables $X,H,Y$. Take $\k=\Z$ again.
The mod $2$ invariant $H$ does not lift to an integral invariant, but $H^2+4XY$ is an integral invariant, and it reduces mod $2$ to the square $H^2$ in $\F_2[X,H,Y]$. This illustrates power reductivity with modules that are not flat, and the
strong link between integral and modular invariants.

\subsubsection{}
Consider the group $U$ of $2\times 2$ upper triangular matrices with diagonal $1$: 
this is just an additive group. 
Let it act on $M$ with basis $\{x,y\}$ by linear substitutions: $u(a)$ sends $x,y$ respectively to $x,ax+y$.
Sending $x$ to $0$ defines $M\to L$, and since $(S^*M)^{U}=\k[x]$, power reductivity fails.

\subsection{Equivalence of power reductivity with property (Int)}\label{power Int:subsec}
Following \cite{vdK reductive}, we say that a group $G$ satisfies (Int) if 
$(A/J)^G$ is integral over the image of $A^G$ for every $A$ and $J$ with $G$ action.
Note that if $(A/J)^G$ is a Noetherian $A^G$-module (compare Theorem \ref{module fg:theorem}),
it must be integral over the image of $A^G$. 
As explained in \cite[Theorem 2.8]{vdK reductive}, when $\k$ is a field, 
the property (Int) is equivalent to geometric reductivity, 
which is equivalent to power-reductivity by \cite[Lemma 2.4.7 p. 23]{Sp}. 
In general, property (Int) is still equivalent to power-reductivity. 
But geometric reductivity in the sense of \cite{Se} looks too weak. 
\begin{Pro}
An algebraic group $G$ has property (Int) if, and only if, it is power-reductive.
\end{Pro}
\begin{proof}
By Proposition \ref{power:Pro}, power reductivity implies property (Int).
We prove the converse.
Let $\phi: M\to L$ be as in the formulation of power reductivity in Definition \ref{power reductive:Def}.
Choose a generator $b$ of $L$.
Property (Int) gives a polynomial $t^n+a_1t^{n-1}+\cdots+a_n$ with $b$ as root, 
and with $a_i$ in the image of $S^*(\varphi):(S^*M)^G\to S^*L$.
As $b$ is homogeneous of degree one, we may assume $a_i\in S^i\varphi( (S^iM)^G)$. 
Write $a_i$ as $r_ib^{i}$ with $r_i\in\k$.
Put $r=1+r_1+\cdots r_n$. Then $rb^n=0$, and $r^{(n-1)!}$ annihilates $b^{n!}$.
Since $a_i^{n!/i}$ lies in the image of $S^{n!}\varphi$ : $(S^{n!}M)^G\to S^{n!}L$,
the cokernel of this map is annihilated by $r_i^{n!/i}$. 
Together $r^{(n-1)!}$ and the $r_i^{n!/i}$ generate the unit ideal. So the cokernel vanishes.
\end{proof}
\begin{Exm} Let $G$ be a
finite group, viewed as an algebraic group over $\k$. Then $A$ is integral over $A^G$,
because $a$ is a root of $\prod_{g\in G}(x-g(a))$. (This goes back to
Emmy Noether \cite{noether}.)
Property (Int) follows easily. Hence $G$ is power reductive.
\end{Exm}
\section{Mumford's conjecture over an arbitrary base}\label{Mumford:section}
This section deals with the following generalization of the Mumford conjecture.
\begin{theorem}[Mumford conjecture]\label{Mum:theorem}
A Chevalley group scheme is power-reductive for every base.
\end{theorem}
By a Chevalley group scheme over $\Z$, we mean a connected split reductive algebraic $\Z$-group $G_\Z$,
and, by a Chevalley group scheme over a ring $\k$, we mean an algebraic $\k$-group $G=G_\k$ obtained by base change from such a $G_\Z$.

\begin{comment}
To prove power reductivity of a group scheme over some base ring $R$, one may first take a faithfully flat extension of $R$.
So some twisted forms of Chevalley group schemes are covered as well. Compare the discussion in \cite[p. 239]{Se}.
\end{comment}

We want to establish the following:
\begin{Property}
Let $\k$ be a commutative ring.
Let $L$ be a cyclic $\k$-module with trivial $G$-action. 
Let $M$ be a rational $G$-module, and let $\varphi$ be a $G$-module map from $M$ onto $L$.
Then there is a positive integer $d$ such that the $d$-th symmetric power of $\varphi$ induces a surjection:
\[
(S^d M)^G\to S^d L . 
\]
\end{Property}
\subsection{Reduction to local rings}
We first reduce to the case of a local ring.
For each positive integer $d$, consider the ideal in $\k$ formed by those scalars which are hit by an invariant in $(S^d M)^G$, and let:
\[
\I_d(\k):=\{x\in \k\ |\ \exists m\in\N,\ x^m.S^d L \subset S^d\varphi((S^d M)^G)\}
\]
be its radical. Note that these ideals form a monotone family: if $d$ divides $d'$, then $\I_d$ is contained in $\I_{d'}$.
We want to show that $\I_d(\k)$ equals $\k$ for some $d$. 
To that purpose, it is enough to prove that for each maximal ideal $\M$ in $\k$, 
the localized $\I_{d}(\k)_{(\M)}$ equals the local ring $\k_{(\M)}$ for some $d$. 
Notice that taking invariants commutes with localization. 
Indeed the whole Hochschild complex does and localization is exact.
As a result, the localized $\I_{d}(\k)_{(\M)}$ is equal to the ideal $\I_d(\k_{(\M)})$.
This shows that it is enough to prove the property for a local ring $\k$. 
\par\medskip
For the rest of this proof, $\k$ denotes a local ring with residual characteristic $p$.
\subsection{Reduction to cohomology}
As explained in Section \ref{reduction}, we may assume that $G$ is semisimple simply connected.
Replacing $M$ if necessary by a submodule that still maps onto $L$, we may assume that $M$ is finitely generated.
\par
We then reduce the desired property to cohomological algebra. 
To that effect, if $X$ is a $G$-module, consider the evaluation map on the identity $\id_X$:
$\hom_\k(X,X)^{\#}\to \k $
(we use $V^\#$ to indicate the dual module $\hom_\k(V,\k)$ of a module $V$).
If $X$ is $\k$-free of finite rank $d$, then $S^d (\hom_\k(X,X)^{\#})$ contains the determinant.
The determinant is $G$-invariant, and its evaluation at $\id_X$ is equal to $1$. 
Let $b$ a $\k$-generator of $L$ and consider the composite:
\[
\psi : \hom_\k(X,X)^{\#}\to \k \to \k.b=L .
\]
Its $d$-th power $S^d\psi$ sends the determinant to $b^d$.
Suppose further that $\psi$ lifts to $M$ by a $G$-equivariant map. 
Then, choosing $d$ to be the $\k$-rank of $X$, the $d$-th power of 
the resulting map $S^d (\hom_\k(X,X)^{\#})\to S^dM$ sends the determinant to a $G$-invariant in $S^dM$,
which is sent to $b^d$ through $S^d\varphi$. 
This would establish the property.
\[
\xymatrix{
& \hom_\k(X,X)^{\#} \ar[d]^{\psi} \ar@{-->}[dl]\\
M \ar[r]_{\varphi} &L 
}
\]
\par
The existence of a lifting would follow from the vanishing of the extension group:
\[
\ext_G^1((\hom_\k(X,X)^{\#},\ker\varphi),
\]
or, better, from acyclicity, {\it i.e.}\ the vanishing of all positive degree Ext-groups.
\par
Inspired by the proof of the Mumford conjecture in \cite[(3.6)]{CPSvdK}, we choose $X$ to be an adequate Steinberg module. 
To make this choice precise, we need notations, essentially borrowed from \cite{CPSvdK,AK}.
\subsection{Notations}\label{notations}
We decide as in \cite{J}, and contrary to \cite{vdK book} and \cite{CPSvdK},
that the roots of the standard Borel subgroup $B$ are negative. 
The opposite Borel group $B^+$ of $B$ will thus have positive roots. 
We also fix a Weyl group invariant inner product on the weight lattice $X(T)$. 
Thus we can speak of the length of a weight.
\par
For a weight $\lambda$ in the weight lattice, 
we denote by $\lambda$ as well the corresponding one-dimensional rational $B$-module
(or sometimes $B^+$-module), and by $\nabla_{\lambda}$ the costandard module (Schur module) $\ind_B^G\lambda$ induced from it. 
Dually, we denote by $\Delta_{\lambda}$ the standard module (Weyl module) of highest weight $\lambda$.
So $\Delta_{\lambda}=\ind_{B^+}^G(-\lambda)^{\#}$.
We shall use that, over $\Z$, these modules are $\Z$-free \cite[II Ch. 8]{J}.
\par
We let $\rho$ be half the sum of the positive roots of $G$. It is also the sum of the fundamental weights. As $G$ is simply connected, the fundamental weights are weights of $B$.
\par
Let $p$ be the characteristic of the residue field of the local ring $\k$.
When $p$ is positive, for each positive integer $r$, we let the weight $\sigma_r$ be $(p^r-1)\rho$.
When $p$ is $0$, we let $\sigma_r$ be $r\rho$.
Let $St_r$ be the $G$-module $\nabla_{\sigma_r}=\ind_B^G\sigma_r$; it is a usual Steinberg module when $\k$ is a field of positive characteristic.
\subsection{}
We shall use the following combinatorial lemma:
\begin{Lem}\label{r:Lem}
Let $R$ be a positive real number. If $r$ is a large enough integer, for all weights $\mu$ of length less than $R$, $\sigma_r+\mu$ is dominant.
\end{Lem}
So, if $r$ is a large enough integer to satisfy the condition in Lemma \ref{r:Lem},
for all $G$-modules $M$ with weights that have length less than $R$,
all the weights in $\sigma_r\ot M$ are dominant.
Note that in the preceding discussion, the $G$-module $M$ is finitely generated. 
Thus the weights of $M$, and hence of $\ker\varphi$, are bounded. 
Thus, Theorem \ref{Mum:theorem} is implied by the following proposition.
\begin{Pro}\label{ext cancel:Pro}
Let $R$ be a positive real number, and let $r$ be as in Lemma \ref{r:Lem} .
For all local rings $\k$ with characteristic $p$ residue field,
for all $G$-module $N$ with weights of length less than $R$, and for all positive integers $n$:
\[
\ext^n_G((\hom_\k(\St,\St)^{\#}, N)=0\ .
\]
\end{Pro}
\begin{proof}
First, the result is true when $\k$ is a field. Indeed, we have chosen $\St$ to be a self-dual Steinberg module, so, for each positive integer $n$:
$$\ext^n_G((\hom_\k(\St,\St)^{\#}, N)=\h^n(G,\St\ot \St\ot N)=\h^n(B,\St\ot \si_r\ot N).$$
Vanishing follows by \cite[Corollary (3.3')]{CPSvdK} or the proof of \cite[Corollary (3.7)]{CPSvdK}.
\par
Suppose now that $N$ is defined over $\Z$ by a free $\Z$-module, in the following sense: $N=N_{\Z}\ot_{\Z} V$ for a $\Z$-free $G_{\Z}$-module $N_{\Z}$ and a $\k$-module $V$ with trivial $G$ action. We then use the universal coefficient theorem \cite[A.X.4.7]{B} (see also \cite[I.4.18]{J}) to prove acyclicity in this case.
\par
Specifically, let us note $Y_{\Z}:=\hom_{\Z} ((\St)_{\Z},(\St)_{\Z})\ot N_{\Z}$, 
so that, using base change (Proposition \ref{base change:Pro} for $\lambda=\sigma_{r}$):
\[\ext^n_G((\hom_\k(\St,\St)^{\#}, N)=\h^n(G,Y_{\Z}\ot V).\]
This cohomology is computed \cite[II.3]{DG} (see also \cite[I.4.16]{J}) by taking the homology of the Hochschild complex $C(G,Y_{\Z}\ot V)$.
This complex is isomorphic to the complex obtained by tensoring with $V$ the integral Hochschild complex $C(G_{\Z},Y_{\Z})$.
Since the latter is a complex of torsion-free abelian groups, we deduce, by the universal coefficient theorem
applied to tensoring with a characteristic $p$ field $k$,
and the vanishing for the case of such a field, that: $\h^n(G_{\Z},Y_{\Z})\ot{k}=0$, for all positive $n$. 
We apply this when $k$ is the residue field of $\Z_{(p)}$; note that if $p=0$ the residue field $k$ is just $\Q$.
Since the cohomology $\h^n(G_{\Z},Y_{\Z})$ is finitely generated \cite[B.6]{J},
the Nakayama lemma implies that: $\h^n(G_{\Z},Y_{\Z})\ot \Z_{(p)}=0$, for all positive $n$. 
And $\h^n(G_{\Z},Y_{\Z})\ot \Z_{(p)}=\h^n(G_{\Z},Y_{\Z}\ot \Z_{(p)})$ because localization is exact.
The complex $C(G_{\Z},Y_{\Z}\ot \Z_{(p)})$ is a complex of torsion-free $\Z_{(p)}$-modules, we thus can apply the universal coefficient theorem to tensoring with $V$.
The vanishing of $\h^n(G,Y_{\Z}\ot\Z_{(p)}\ot V)=\h^n(G,Y_{\Z}\ot V)$ follows. 
\par
For the general case, we proceed by descending induction on the highest weight of $N$.
To perform the induction, we first choose a total order on weights of length less than $R$, 
that refines the usual dominance order of \cite[II 1.5]{J}. 
Initiate the induction with $N=0$. For the induction step,
consider the highest weight $\mu$ in $N$ and let $N_{\mu}$ be its weight space. 
By the preceding case, we obtain vanishing for $\Delta_{\mu_{\Z}}\ot_{\Z}N_{\mu}$. 
Now, by Proposition \ref{universal Weyl:Pro}, 
$\Delta_{\mu_{\Z}}\ot_{\Z}N_{\mu}$ maps to $N$, and the kernel and the cokernel of this map have lower highest weight. By induction, they give vanishing cohomology. 
Thus $\hom_\k(\St,\St)\ot N$ is in an exact sequence where three out of four terms are acyclic, 
hence it is acyclic.
\end{proof}
This concludes the proof of Theorem \ref{Mum:theorem}.
\begin{comment}
This proof was not effective. If one wants an estimate of the degree $d$ in terms of
the weights that occur in $M$, there are several ways to go about it.
One way is to use the same definition of the Steinberg module $\St$ when the residue characteristic $p$ of the local ring $\k$
is large as when $p$ is zero.
Indeed say $r$ is so large that $r\rho+\mu$ is dominant for every
weight $\mu$ of $M$, and $p$ is so large that $r\rho$ lies in the bottom alcove for
the corresponding affine Weyl group $W_p$. Then $ \ind_{B}^{G}(r\rho)$
is self-dual \cite[II, 2.5, 5.6]{J} and may thus play the part of $\St$.
Note that one finds a degree $d$ that is a power of $r$ rather than a power of
$p$.
Assuming $r$ prime to $p$ (is this automatic? Increase $r$ by one if needed.) one learns
that in fact degree $d=1$ must work, by the proof of
Proposition \ref{co t:Prop}.

There is a reason that $d=1$ works here.
If $p=0$ or if $p$ is large with respect to the weights of $N$, one may also show
that the conclusion of 
Proposition 
{\ref{ext cancel:Pro}} holds for $r=0$, corresponding with $d=1$.
The reason is that the relevant $\Delta_{\mu_\Z}\otimes\Z_{(p)}$
are isomorphic with
$\nabla_{\mu_\Z}\otimes\Z_{(p)}$ and thus have vanishing cohomology.
This can then be used as in the proof of Proposition 
{\ref{ext cancel:Pro}}.
But do we care enough to tell the reader?
\end{comment}
\begin{commentVF}
But why would we emphasize large primes ?
\end{commentVF}
\subsection{Reduction to simply connected group schemes}\label{reduction}
Let $Z_\Z$ be the center of $G_\Z$ and let $Z$ be the corresponding subgroup of $G$. 
It is a diagonalisable group scheme, so $M^Z\to L$ is also surjective. 
We may replace $M$ with $M^Z$ and $G$ with $G/Z$, in view of the general formula $M^G=(M^Z)^{G/Z}$, see \cite[I 6.8(3)]{J}. 
So now $G$ has become semisimple, but of adjoint type rather than simply connected type. 
So choose a simply connected Chevalley group scheme $\tilde G_\Z$ with center $\tilde Z_\Z$ so that $\tilde G_\Z/\tilde Z_\Z=G_\Z$. 
We may now replace $G$ with $\tilde G$. 
\begin{Rem}
Other reductions are possible, to enlarge the supply of power reductive algebraic groups. 
For instance, if $G$ has a normal subgroup $N$ so that both $N$ and $G/N$ are power reductive, then so is $G$
(for a proof, use Remark \ref{pow inv:rem}). 
And if $\k\to R$ is a faithfully flat extension so that $G_R $ is power reductive,
then $G$ is already power reductive.
So twisted forms are allowed, compare the discussion in \cite[p. 239]{Se}.
\end{Rem}
\begin{comment}Preceding remark added. It may not be the right place.
One should also prove it. How? For the $G$, $N$ statement
use as in Proposition \ref{power:Pro} the formulation that, for a (power?)
surjective $G$-map $\phi:A\to B$ of $\k$-algebras, the induced map on invariants $A^G\to B^G$ is power-surjective.
Note that a finite group is power reductive because it satisfies (Int). 
Should we say that? More generally, finite generation implies power reductivity.
We have to decide where to put things. 
The equivalence between (Int) and power reductivity works over any base.
\end{comment}
%
\section{Generalities}\label{generalities:section}
This section collects known results over an arbitrary base, their proof, and correct proofs of known results over fields, for use in the other sections. 
The part up to subsection \ref{Grosshans notations:subsection} is used, and referred to, in the previous section.
\subsection{Notations}
Throughout this paper, we let $G$ be a semisimple Chevalley group scheme over the commutative ring $\k$.
We keep the notations of Section \ref{notations}. In particular, the standard parabolic $B$ has negative roots. 
Its standard torus is $T$, its unipotent radical is $U$. The opposite Borel $B^+$ has positive roots and its unipotent radical is $U^+$.
For a standard parabolic subgroup $P$ its unipotent radical is $R_u(P)$.
For a weight $\lambda$ in $X(T)$,
$\nabla_{\lambda}=\ind_B^G\lambda$ and $\Delta_{\lambda}=\ind_{B^+}^G(-\lambda)^{\#}$.
\subsection{}
We first recall base change for costandard modules.
\begin{Pro}\label{base change:Pro}
Let $\lambda$ be a weight, and denote also by $\lambda=\lambda_{\Z}\ot\k$ the $B$-module $\k$ with action by $\lambda$. For any ring $\k$, there is a natural isomorphism:
\[
\ind_{B_\Z}^{G_\Z}\lambda_\Z\ot \k\cong\ind_B^G\lambda
\]
In particular, $\ind_B^G\lambda$ is nonzero if and only if $\lambda$ is dominant.
\end{Pro}
\begin{proof}
First consider the case when $\lambda$ is not dominant. 
Then $\ind_B^G\lambda$ vanishes when $\k$ is a field \cite[II.2.6]{J}, so both $\ind_{B_\Z}^{G_\Z}\lambda_\Z$ and the torsion in $R^1\ind_{B_\Z}^{G_\Z}\lambda_\Z$ must vanish. Then $\ind_B^G\lambda$ vanishes as well for a general $\k$ by the universal coefficient theorem.
\par
In the case when $\lambda$ is dominant, $R^1\ind_{B_\Z}^{G_\Z}\lambda_\Z$ vanishes by Kempf's theorem \cite[II 8.8(2)]{J}. Thus, by \cite[I.4.18b)]{J}: 
$\ind_{B_\Z}^{G_\Z}\lambda_\Z\ot \k\cong\ind_B^G\lambda$.
\end{proof}
\begin{Pro}[Tensor identity for weights]\label{gti:Pro}
Let $\lambda$ be a weight, and denote again by $\lambda$ the $B$-module $\k$ with action by $\lambda$. 
Let $N$ be a $G$-module. 
There is a natural isomorphism:
\[
\ind_B^G(\lambda\ot N)\cong (\ind_B^G \lambda)\ot N .
\]
\end{Pro}
\begin{Rem}
The case when $N$ is $\k$-flat is covered by \cite[I.4.8]{J}. We warn the reader against Proposition I.3.6 in the 1987 first edition of the book.
Indeed, suppose we always had\break
$\ind_B^G(M\ot N)\cong (\ind_B^G M)\ot N$. 
Take $\k=\Z$ and $N=\Z/p\Z$. The universal
coefficient theorem would then imply that $R^1\ind_B^G M$ never has torsion.
Thus $R^i\ind_B^G M$ would never have torsion for positive $i$.
It would make \cite[Cor. 2.7]{A} contradict the Borel--Weil--Bott theorem.
\end{Rem}
\begin{commentVF}
add a counterexample
\end{commentVF}
\begin{proof}
Recall that for a $B$-module $M$ one may define $\ind_B^G(M)$ as $(\k[G]\ot M)^B$, where
$\k[G]\ot M$ is viewed as a $G\x B$-module with $G$ acting by left translation on $\k[G]$,
$B$ acting by right translation on $\k[G]$, and $B$ acting the given way on $M$.
Let $N_\triv$ denote $N$ with trivial $B$ action.
There is a $B$-module isomorphism $\psi:\k[G]\ot\lambda\ot N\to \k[G]\ot\lambda\ot N_\triv$,
given in non-functorial notation by: 
\[
\psi(f\ot1\ot n): x\mapsto f(x)\ot1\ot xn
.\]
So $\psi$ is obtained by first applying the comultiplication $N\to \k[G]\ot N$, then the
multiplication $\k[G]\ot \k[G]\to \k[G]$.
It sends $(\k[G]\ot\lambda\ot N)^B$ to 
$(\k[G]\ot\lambda\ot N_\triv)^B=(\Z[G_\Z]\ot_\Z \lambda_\Z \ot_\Z N_\triv)^B$.
Now recall from the proof of Proposition \ref{base change:Pro} that the torsion in $R^1\ind_{B_\Z}^{G_\Z}\lambda_\Z$ vanishes.
By the universal coefficient theorem we get 
that $(\Z[G_\Z]\ot_\Z \lambda_\Z \ot_\Z N_\triv)^B$ equals $(\k[G]\ot \lambda)^B \ot N_\triv$. 
To make these maps into $G$-module maps, one must use the given $G$-action on $N$ as the action
on $N_\triv$. So $B$ acts on $N$, but not $N_\triv$, and for $G$ it is the other way around.
One sees that $(\k[G]\ot \lambda)^B \ot N_\triv$ is just $(\ind_B^G \lambda)\ot N$.
\end{proof}
\begin{Pro}\label{dominant hw:Pro}
For a $G$-module $M$, there are only dominant weights in $M^{U^+}$.
\end{Pro}
\begin{proof}
Let $\lambda$ be a nondominant weight. Instead we show that $-\lambda$ is no weight of $M^{U}$,
or that $\hom_B(-\lambda,M)$ vanishes. 
By the tensor identity of Proposition \ref{gti:Pro}:
$
\hom_B(-\lambda,M)=\hom_B(k,\lambda\ot M)=\hom_G(k,\ind_B^G(\lambda\ot M))=\hom_G(k,\ind_B^G\lambda\ot M)
$
which vanishes by Proposition \ref{base change:Pro}.
\end{proof}
\begin{Pro}\label{evaluation:Pro}
Let $\lambda$ be a dominant weight. The restriction (or evaluation) map 
$\ind_B^G\lambda\to\lambda$ to the weight space of weight $\lambda$ is a $T$-module isomorphism.
\end{Pro}
\begin{proof}
Over fields of positive characteristic this is a result of Ramanathan \cite[A.2.6]{vdK book}.
It then follows over $\Z$ by the universal coefficient theorem applied to the complex
$\ind_{B_\Z}^{G_\Z}\lambda_\Z\to \lambda_\Z\to 0$.
For a general $\k$, apply proposition \ref{base change:Pro}.
\end{proof}
\begin{Pro}[Universal property of Weyl modules]\label{universal Weyl:Pro}
Let $\lambda$ be a dominant weight.
For any $G$-module $M$, there is a natural isomorphism
\[
\hom_G(\Delta_\lambda,M)\cong \hom_{B^+}(\lambda,M).
\]
In particular, if $M$ has highest weight $\lambda$,
then there is a natural map from $\Delta_{\lambda_{\Z}}\ot_{\Z}M_{\lambda}$ to $M$, 
its kernel has lower weights, and $\lambda$ is not a weight of its cokernel.
\end{Pro}
\begin{proof}
By the tensor identity Proposition \ref{gti:Pro}:
$\ind_{B^+}^G(-\lambda\ot M)\cong\ind_{B^+}^G(-\lambda)\ot M$.
Thus
$
\hom_G(\Delta_\lambda,M)=\hom_G(\k,\ind_{B^+}^G(-\lambda)\ot M)=
\hom_{B^+}(\k,-\lambda\ot M)=\hom_{B^+}(\lambda,M).
$
If $M$ has highest weight $\lambda$, then $M_\lambda=\hom_{B^+}(\lambda,M)$.
Tracing the maps, the second part follows from Proposition \ref{evaluation:Pro}.
\end{proof}
\subsection{Notations}\label{Grosshans notations:subsection}
We now recall the notations from \cite[\S 2.2]{vdK coh}. 
Let the \emph{Grosshans height function} $\hgt:X(T)\to \Z$ be defined by:
\[
\hgt \gamma=\sum_{\alpha>0}\langle \gamma,\alpha^\vee\rangle
.\]
For a $G$-module $M$, let $M_{\leq i}$ denote the largest $G$-submodule with weights $\lambda$ that all satisfy: $\hgt \lambda\leq i$. Similarly define $M_{<i}=M_{\leq i-1}$.
For instance, $M_{\leq0}=M^G$.
We call the filtration 
$$
0\subseteq M_{\leq0}\subseteq M_{\leq1}\cdots
$$
the \emph{Grosshans filtration}, and we call its associated graded the \emph{Grosshans graded} $\gr M$ of $M$. 
We put: $\hull_\nabla(\gr M)=\ind_{B}^GM^{U^+}.$
\par
Let $A$ be a commutative $\k$-algebra on which $G$ acts rationally through $\k$-algebra automorphisms.
The Grosshans graded algebra $\gr A$ is given in degree $i$ by: 
\[
\gr_i A=A_{\leq i}/A_{< i}.
\]
\subsection{Erratum}\label{correcting the book}
When $\k$ is a field, one knows that $\gr A$ embeds in a good filtration hull, which Grosshans calls $R$ in \cite{Grosshans contr}, and which we call $\hull_\nabla(\gr A)$:
$$\hull_\nabla(\gr A)=\ind_{B}^GA^{U^+}.$$
When $\k$ is a field of positive characteristic $p$, it was shown by Mathieu \cite[Key Lemma 3.4]{Mathieu G} that this inclusion is power-surjective: for every $b\in \hull_\nabla(\gr A)$, there is an $r$ so that $b^{p^r}$ lies in the subalgebra $\gr A$. 
\par
This result's exposition in \cite[Lemma 2.3]{vdK coh} relies on \cite[Sublemma A.5.1]{vdK book}.
Frank Grosshans has pointed out that the proof of this sublemma 
is not convincing beyond the reduction to the affine case. 
Later A.~J.~de Jong actually gave a counterexample to the reasoning. 
The result itself is correct and has been used by others. 
As power surjectivity is a main theme in this paper, we take the opportunity to give a corrected treatment.
Mathieu's result will be generalized to an arbitrary base $\k$ in Section \ref{Grosshans:section}.
\begin{Pro}Let $\k$ be an algebraically closed field of characteristic $p>0$. 
Let both $A$ and $B$ be commutative $\k$-algebras of finite type over $\k$, with $B$ finite over $A$. 
Put $Y=\Spec(A)$, $X=\Spec(B)$. Assume $X\to Y$ gives a bijection between $\k$ valued points. 
Then for all $b\in B$ there is an $m$ with $b^{p^m}\in A$.
\end{Pro}
\begin{proof}
The result follows easily from \cite[Lemma 13]{Mathieu compos 69}.
We shall argue instead by induction on the Krull dimension of $A$.
\par
Say $B$ as an $A$-module is generated by $d$ elements $b_1$, \dots, $b_d$.
Let $\p_1$, \dots $\p_s$ be the minimal prime ideals of $A$.
\par
Suppose we can show that for every $i$, $j$ we have $m_{i,j}$
so that $b_j^{p^{m_{i,j}}}\in A+\p_iB $.
Then for every $i$ we have $m_i$ so that
$b^{p^{m_{i}}}\in A+\p_iB $ for every $b\in B$.
Then $b^{p^{m_{1}+\cdots m_{s}}}\in A+\p_1\cdots\p_s B$ for every $b\in B$.
As $\p_1\cdots\p_s$ is nilpotent, one finds $m$ with $b^{p^m}\in A$ for all 
$b\in B$. The upshot is that it suffices to
prove the sublemma for the inclusion $A/\p_i \subset B/\p_iB$.
[It is an inclusion because there is a prime ideal $\q_i$ in $B$ with $A\cap
\q_i=\p_i$.]
Therefore we further assume that $A$ is a domain.
\par
Let $\rad$ denote the nilradical of $B$.
If we can show that for all $b\in B$ there is $m$ with $b^{p^m}\in A+\rad$,
then clearly we can also find a $u$ with $b^{p^u}\in A$. So we may as well
replace $A\subset B$ with $A\subset B/\rad$ and
assume that $B$ is reduced. But then at least one component of
$\Spec(B)$ must map onto $\Spec(A)$, so bijectivity implies there is only
one component. In other words, $B$ is also a domain.
\par
Choose $t$ so that the field extension $\Frac(A)\subset \Frac(AB^{p^t})$ is
separable. (So it is the separable closure of $\Frac(A)$ in $\Frac(B)$.)
As $X\to \Spec(AB^{p^t})$ is also bijective, we have that 
$\Spec(AB^{p^t})\to\Spec(A)$ is bijective. It clearly suffices to prove
the proposition for $A\subset AB^{p^t}$. So we replace $B$ with $AB^{p^t}$ and
further assume that $\Frac(B)$ is separable over $\Frac(A)$.
\par
Now $X\to Y$ has a degree which is the degree of the
separable field extension. There is a dense subset $U$ of $Y$ so that
this degree is the number of elements
in the inverse image of a point of $U$. 
[Take a primitive element of the field extension, 
localize to make its minimum polynomial monic over $A$,
invert the discriminant.]
Thus the degree must be one because of bijectivity.
So we must now have that $\Frac(B)= \Frac(A)$.
\par
Let 
$\con=\{\;b\in B\mid bB \subset A\;\}$ be the conductor of $A \subset B$.
We know it is nonzero. 
If it is the unit ideal then we are done. Suppose it is not.
By induction applied to $A/\con\subset B/\con$ (we need 
the induction hypothesis for the original problem without any of the
intermediate simplifications)
we have that for each $b\in B$ there is
an $m$ so that $b^{p^m}\in A+\con=A$. 
\end{proof}
\subsection{}
This subsection prepares the ground for the proof of the theorems in Section \ref{Grosshans:section}.
We start with the ring of invariants $\k[G/U]$ of the action of $U$ by right translation on $\k[G]$.
\begin{Lem}\label{fg multicone:Lem}
The $\k$-algebra $\k[G/U]$ is finitely generated.
\end{Lem}
\begin{proof}
We have:
\[
k[G/U]=\bigoplus_{\lambda\in X(T)}k[G/U]_{-\lambda}=\bigoplus_{\lambda\in X(T)}(k[G]\ot{\lambda})^B=\bigoplus_{\lambda\in X(T)}\nabla_{\lambda}
.\]
By Proposition \ref{base change:Pro},
this equals the sum $\oplus_{\lambda}\nabla_\lambda$ over dominant weights $\lambda$ only. 
When $G$ is simply connected, every fundamental weight is a weight, and the monoid of dominant $\lambda$ is finitely generated. 
In general, some multiple of a fundamental weight is in $X(T)$ and there are only finitely many dominant weights modulo these multiples.
So the monoid is still finitely generated
by Dickson's Lemma \cite[Ch. 2 Thm. 7]{CLO}.
The maps $\nabla_\lambda\ot \nabla_\mu\to \nabla_{\lambda+\mu}$ are surjective for dominant $\lambda$, $\mu$,
because this is so over $\Z$, by base change and surjectivity for fields \cite[II, Proposition 14.20]{J}. This implies the result.
\end{proof}
In the same manner one shows:
\begin{Lem}\label{fg hull:Lem}
If the $\k$-algebra $A^U$ is finitely generated, so is $\hull_\nabla\gr A=\ind_B^GA^{U^+}$. 
\end{Lem}
\begin{proof}
Use that $A^{U^+}$ is isomorphic to $A^U$ as $\k$-algebra. 
\end{proof} 
\begin{Lem}\label{U inv fg:Lem}
Suppose $\k$ is Noetherian. If the $\k$-algebra $A$ with $G$ action is finitely generated, then so is $A^U$. 
\end{Lem}
\begin{proof}
By the transfer principle \cite[Ch. Two]{G}:
\[
A^U=\hom_U(k,A)=\hom_G(k,\ind_U^GA)=(A\ot k[G/U])^G
.\]
Now apply Lemma \ref{fg multicone:Lem} and Theorem \ref{FFT:theorem}.
\end{proof} 
\begin{Lem}\label{embed:Lem}
If $M$ is a $G$-module, there is a natural injective map
\[
\gr M\hookrightarrow\hull_\nabla(\gr M)=\ind_{B}^GM^{U^+}
.\]
\end{Lem}
\begin{proof}
By Lemma \ref{dominant hw:Pro}, the weights of $M^{U^+}$ are dominant. 
If one of them, say $\lambda$, has Grosshans height $i$,
the universal property of Weyl modules (Proposition \ref{universal Weyl:Pro}) shows that $(M^{U^+})_\lambda$ is contained in a $G$-submodule with weights that do not have a larger Grosshans height. 
So the weight space $(M^{U^+})_\lambda$ is contained in $M_{\leq i}$, but not $M_{<i}$. 
We conclude that the $T$-module $\oplus_i(\gr_iM)^{U^+}$
may be identified with the $T$-module $M^{U^+}$.
It remains to embed $\gr_iM$ into $\ind_B^G(\gr_iM)^{U^+}$.
The $T$-module projection of $\gr_iM$ onto $(\gr_iM)^{U^+}$ may be viewed as a $B$-module map, 
and then, it induces a $G$-module map $\gr_iM\to \ind_B^G((\gr_iM)^{U^+})$, 
which restricts to an isomorphism on $(\gr_iM)^{U^+}$ by Proposition \ref{evaluation:Pro}. 
So its kernel has weights with lower Grosshans height, and must therefore be zero. 
\end{proof}
In the light of Lemma \ref{embed:Lem}, one may write:
\begin{Def}
A $G$-module $M$ has \emph{good Grosshans filtration} if
the embedding of $ \gr M$ into $\hull_\nabla(\gr M)$ is an isomorphism.
\end{Def}
It is worth recording the following characterization, just like in the case where $\k$ is a field.
\begin{Pro}[Cohomological criterion]\label{coh crit:Pro}
For a $G$-module $M$, the following are equivalent.
\begin{enumerate}
\item $M$ has good Grosshans filtration.
\item $\h^1(G,M\ot \k[G/U])$ vanishes.
\item $\h^n(G,M\ot \k[G/U])$ vanishes for all positive $n$.
\end{enumerate}
\end{Pro}
\begin{proof}
Let $M$ have good Grosshans filtration. We have to show that $M\ot \k[G/U]$ is acyclic.
First, for each integer $i$, $\gr_i M\ot \k[G/U]$ is a direct sum of modules of the form $\ind_B^G\lambda\ot\ind_B^G\mu\ot N$, where $G$ acts trivially on $N$.
Such modules are acyclic by \cite[B.4]{J} and the universal coefficient theorem.
As each $\gr_i M\ot \k[G/U]$ is acyclic, so is each $M_{\leq i}\ot \k[G/U]$, and thus $M\ot \k[G/U]$ is acyclic.
\par
Conversely, suppose that $M$ does not have good Grosshans filtration. Choose $i$ so that $M_{<i}$ has good Grosshans filtration, but $M_{\leq i}$ does not.
Choose $\lambda$ so that $\hom(\Delta_\lambda,\hull(\gr_i M)/\gr_i M)$ is nonzero. Note that $\lambda$ has Grosshans height below $i$. As $\hom(\Delta_\lambda,\hull(\gr_i M))$ vanishes, 
$\ext^1_{G}(\Delta_\lambda,\gr_i M)=\h^1(G,\gr_i M\ot\nabla_\lambda)$ does not.
Since $M_{<i}\ot \k[G/U]=\oplus_{\mu \mbox{ dominant}}M_{<i}\ot\nabla_\mu$ is acyclic, 
$\h^1(G,M_{\leq i}\ot\nabla_\lambda)$ is nonzero as well. 
Now use that $\hom(\Delta_\lambda,M/M_{\leq i})$ vanishes, and conclude that $\h^1(G,M\ot \k[G/U])$ does not vanish.
\end{proof}

\section{Grosshans graded, Grosshans hull and powers}\label{Grosshans:section}
\subsection{}
When $G$ is a semisimple group over a field $\k$,
Grosshans has introduced a filtration on $G$-modules.
As recalled in Section \ref{Grosshans notations:subsection}, it is the filtration associated to the function defined on $X(T)$ by:
$\hgt \gamma=\sum_{\alpha>0}\langle \gamma,\alpha^\vee\rangle$.
Grosshans has proved some interesting results about its associated graded,
when the $G$-module is a $\k$-algebra $A$ with rational $G$ action. 
We now show how these results generalize to an arbitrary Noetherian base $\k$,
and we draw some conclusions about $\h^*(G,A)$. 
All this suggests that the finite generation conjecture of \cite{vdK coh} (see also \cite{vdK reductive}) deserves to be investigated in the following generality. 
\begin{Pb}\label{vdk conj:Pb}
Let $\k$ be a Noetherian ring and let $G$ be a Chevalley group scheme over $\k$.
Let $A$ be a finitely generated commutative $\k$-algebra on which $G$ acts rationally through algebra automorphisms. 
Is the cohomology ring $\h^*(G,A)$ a finitely generated $\k$-algebra?
\end{Pb}
Let $\k$ be an arbitrary commutative ring.
\begin{theorem}[Grosshans hull and powers]\label{hull and powers:theorem}
The natural embedding of $\gr A$ in $\hull_\nabla(\gr A)$ is power-surjective.
\end{theorem}
This will then be used to prove:
\begin{theorem}[Grosshans hull and finite generation]\label{hull and finite generation:theorem}
If the ring $\k$ is Noetherian, then the following are equivalent.
\begin{enumerate}
\item The $\k$-algebra $A$ is finitely generated;
\item For every standard parabolic $P$, the $\k$-algebra of invariants $A^{R_u(P)}$ is finitely generated;
\item The $\k$-algebra $\gr A$ is finitely generated;
\item The $\k$-algebra $\hull_\nabla(\gr A)$ is finitely generated.
\end{enumerate}
\end{theorem}
\begin{Rem}
Consider a reductive Chevalley group scheme $G$.
As the Grosshans height is formulated with the help of coroots $\alpha^\vee$, 
only the semisimple part of $G$ is relevant for it. 
But of course everything is compatible with the action of the center of $G$ also. 
We leave it to the reader to reformulate our results for reductive $G$. 
We return to the assumption that $G$ is semisimple.
\end{Rem}
\begin{theorem}\label{torsion:theorem}
Let $A$ be a finitely generated commutative $\k$-algebra. If $\k$ is Noetherian,
there is a positive integer $n$ so that:
\[
n\hull_\nabla(\gr A)\subseteq \gr A
.\]
In particular $\h^i(G,\gr A)$ is annihilated by $n$ for positive $i$.
\end{theorem}
This is stronger than the next result.
\begin{theorem}[generic good Grosshans filtration]\label{generic g G:theorem}
Let $A$ be a finitely generated commutative $\k$-algebra. If $\k$ is Noetherian,
there is a positive integer $n$ so that $A[1/n]$ has good Grosshans filtration.
In particular $\h^i(G,A)\ot\Z[1/n]= \h^i(G,A[1/n])$ vanishes for positive $i$.
\end{theorem}
\begin{Rem}
Of course $A[1/n]$ may vanish altogether, as we are allowed to take the characteristic for $n$,
when that is
positive.
\end{Rem}
\begin{theorem}\label{gr power:theorem}
Let $A$ be a finitely generated commutative $\k$-algebra. If $\k$ is Noetherian, for each prime number $p$,
the algebra map $\gr A\to \gr (A/pA)$ is power-surjective.
\end{theorem}

\subsection{}
We start with a crucial special case.
Let $\k=\Z$. Let $\lambda\in X(T)$ be 
dominant.
Let $S'$ be the graded algebra with degree $n$ part:
$$S'_{n}=\nabla_{n\lambda}=\Gamma(G/B,\L(n\lambda)).$$
Let us view $\Delta_\lambda$ as a submodule of $\nabla_\lambda$ with common $\lambda$ weight space (the `minimal admissible lattice' embedded in the `maximal admissible lattice').
Let $S$ be the graded subalgebra generated by $\Delta_\lambda$ in the graded algebra $S'$. 
If we wish to emphasize the dependence on $\lambda$, we write $S'(\lambda)$ for $S'$, $S(\lambda)$ for $S$.
Consider the map 
$$G/B\to \P_{\Z}(\Gamma(G/B,\L(\lambda))^\#)$$ 
given by the `linear system' $\nabla_\lambda$ on $G/B$. The projective scheme $\Proj(S')$ corresponds with its image, which,
by direct inspection, 
is isomorphic to $G/P$,
where $P$ is the stabilizer of the weight space with weight $-\lambda$ of $\nabla_\lambda^\#$.
Indeed that weight space is the image
of $B/B$, compare Proposition \ref{evaluation:Pro} and \cite[II.8.5]{J}.
\begin{comment}Added this line. Did you remove it?
\end{comment}
The inclusion $\phi:S\hookrightarrow S'$ induces a morphism from an open subset of $\Proj(S')$ to $\Proj(S)$.
This open subset is called $G(\phi)$ in \cite[2.8.1]{EGA II}.
\begin{Lem}\label{everywhere:Lem}
The morphism $\Proj(S')\to \Proj(S)$ is defined on all of $G/P=\Proj(S')$.
\end{Lem}
\begin{proof}
As explained in \cite[2.8.1]{EGA II}, the domain $G(\phi)$ contains the principal open subset $D_+(s)$ of $\Proj(S')$ for any $s\in S_1$. 
Consider in particular a generator $s$ of the $\lambda$ weight space of $\nabla_\lambda$. 
It is an element in $S_1$, and, by Lemma \ref{evaluation:Pro}, it generates the free $\k$-module $\Gamma(P/P,\L(\lambda))$. 
Thus, the minimal Schubert variety $P/P$ is contained in $D_+(s)$. 
We then conclude by homogeneity: $s$ is also $U^+$ invariant, so in fact
the big cell $\Omega=U^+P/P$ is contained in $D_+(s)$, and the domain $G(\phi)$ contains the big cell $\Omega$. 
Then it also contains the Weyl group translates $w\Omega$, and thus it contains all of $G/P$. 
\end{proof}
\begin{Lem}\label{integral:Lem}
The graded algebra $S'$ is integral over its subalgebra $S$.
\end{Lem}
\begin{proof}
We also put a grading on the polynomial ring $S'[z]$, by assigning degree one to the variable $z$.
One calls $\Proj(S'[z])$ the projective cone of $\Proj(S')$ \cite[8.3]{EGA II}.
By \cite[8.5.4]{EGA II}, we get from Lemma \ref{everywhere:Lem}
that $\hat{\Phi}:\Proj(S'[z])\to \Proj(S[z])$ is everywhere defined. 
Now by \cite[Th (5.5.3)]{EGA II}, and its proof,
the maps $\Proj(S'[z])\to\Spec \Z$ and $\Proj(S[z])\to\Spec \Z$ are proper and separated, 
so $\hat{\Phi}$ is proper by \cite[Cor (5.4.3)]{EGA II}. 
But now the principal open subset $D_+(z)$ associated to $z$ in $\Proj(S'[z])$ is just $\Spec(S')$,
and its inverse image is the principal open subset associated to $z$ in $\Proj(S[z])$, which is $\Spec(S)$ (compare [EGA II, 8.5.5]).
So $\Spec(S)\to \Spec(S')$ is proper, and $S'$ is a finitely generated $S$-module by \cite[Prop (4.4.2)]{EGA III}.
\end{proof}
\begin{Lem}\label{torsion:Lem}
There is a positive integer $t$ so that $tS'$ is contained in $S$.
\end{Lem}
\begin{proof}
Clearly $S'\ot \Q=S\ot\Q$, so the result follows from Lemma \ref{integral:Lem}.
\end{proof}

Let $p$ be a prime number. Recall from \ref{correcting the book} the result of Mathieu \cite[Key Lemma 3.4]{Mathieu G} that,
for every element $b$ of $S'/pS'$,
there is a positive $r$ so that $b^{p^r}\in (S+pS')/(pS')\subseteq S'/pS'$. 

By Lemma \ref{torsion:Lem} and Proposition \ref{co t:Prop} below this implies
\begin{Lem}\label{universal:Lem}
The inclusion $S\to S'$ is universally power-surjective.
\end{Lem}

\subsection{}
We briefly return to power surjectivity for a general commutative ring $\k$.
\begin{Def}
Let $t$ be a positive integer and let $f:Q\to R$ a $\k$-algebra homomorphism.
We say that $f$ is $t$-power-surjective if for every $x\in R$ there is a power $t^n$ with $x^{t^n}\in f(Q)$.
\end{Def}

\begin{Pro}\label{co t:Prop}
Let $f:Q\to R$ be a $\k$-algebra homomorphism and $Y$ a variable.
\begin{itemize}
\item If $f\otimes \k[Y]: Q[Y]\to R[Y]$ is power-surjective, then $Q\to R/pR$ is $p$-power-surjective for every prime $p$;
\item Assume $t$ is a positive integer such that
$tR\subset f(Q)$. If $Q\to R/pR$ is $p$-power-surjective for every prime $p$ dividing $t$,
then $f$ is universally power-surjective.
\end{itemize}
\end{Pro}
\begin{proof}
First suppose $f\otimes \k[Y]: Q[Y]\to R[Y]$ is power-surjective.
Let $x\in R/pR$. We have to show that $x^{p^n}$ lifts to $Q$ for some $n$.
As $R[Y]\to (R/pR)[Y]$ is surjective, the composite $Q[Y]\to (R/pR)[Y]$ is also power-surjective.
Choose $n$ prime to $p$ and $m$ so that $(x+Y)^{np^m}$ lifts to $Q[Y]$. 
Rewrite $(x+Y)^{np^m}$ as $(x^{p^m}+Y^{p^m})^{n}$ and note that the coefficient $nx^{p^m}$ of $Y^{(n-1)p^m}$ must lift to $Q$.
Now use that $n$ is invertible in $\k/p\k$. 

Next suppose $tR\subset f(Q)$ and
$Q\to R/pR$ is $p$-power-surjective for every prime $p$ dividing $t$. 
Let $C$ be a $\k$-algebra. We have to show that $f\otimes C:Q\otimes C \to R\otimes C$
is power-surjective. Since $f\otimes C:Q\otimes C \to R\otimes C$ satisfies all the conditions that $f:Q\to R$ does,
we may as well simplify notation and suppress $C$. For $x\in R$ we have to show that some power lifts to $Q$.
By taking repeated powers we can get $x$ in $f(Q)+pR$ for every prime $p$ dividing $t$.
So if $p_1$,\ldots,$p_m$ are the primes dividing $t$, we can arrange that $x$ lies in the intersection of the 
$f(Q)+p_iR$, which is $f(Q)+p_1\cdots p_mR$. Now by taking repeated $p_1\cdots p_m$-th powers, one pushes it into $f(Q)+(p_1\cdots p_m)^nR$ for any positive $n$, eventually into $f(Q)+tR\subseteq f(Q)$.
\end{proof}

\subsection{}
We come back to the 
$\k$-algebra $A$, and consider the inclusion
$\gr A\hookrightarrow\hull_\nabla(\gr A)$, as in Theorem \ref{hull and powers:theorem}.
\begin{Not}
Let $\lambda$ be a dominant weight and let $b\in A^{U^+}$ be a weight vector of weight $\lambda$.
Then we define $\psi_b:S'(\lambda)\ot \k\to \hull_\nabla(\gr A)$ as the algebra map induced by the
$B$-algebra map $S'(\lambda)\ot \k\to A^{U^+}$ which sends the generator (choose one) of the
$\lambda$ weight space of $\nabla_\lambda$ to $b$.
\end{Not}
\begin{Lem}\label{powers weight:Lem}
For each $c$ in the image of $\psi_b$, there is a positive integer $s$ so that $c^s\in \gr A$. 
\end{Lem}
\begin{proof}
The composite of $S\ot\k\to S'\ot \k$ with $\psi_b$ factors through $\gr A$, so this follows from Lemma \ref{universal:Lem}.
\end{proof}
\begin{proof}[Proof of Theorem \ref{hull and powers:theorem}]
For every $b\in \hull_\nabla(\gr A)$, there are $b_1$,\ldots, $b_s$ of respective weights
$\lambda_1$,\ldots, $\lambda_s$
so that $b$ lies in the image of
$\psi_{b_1}\ot\cdots\ot\psi_{b_s}$. 
As $\bigotimes_{i=1}^s S({\lambda_i})\to \bigotimes_{i=1}^sS'({\lambda_i})$ is universally power-surjective 
by lemma \ref{universal:Lem}, lemma \ref{powers weight:Lem} easily extends to tensor products.
\end{proof}
\begin{Lem}\label{hull fg gr fg:Lem}
Suppose $\k$ is Noetherian. If $\hull_\nabla(\gr A)$ is a finitely generated $\k$-algebra, so is $\gr A$.
\end{Lem}
\begin{proof}
Indeed, $\hull_\nabla(\gr A)$ is integral over $\gr A$ by Theorem \ref{hull and powers:theorem}.
Then it is integral over a finitely generated subalgebra of $\gr A$, and it is a Noetherian module
over that subalgebra.
\end{proof}
\begin{Lem}\label{gr fg fg:Lem}
If $\gr A$ is finitely generated as a $\k$-algebra, then so is $A$.
\end{Lem}
\begin{proof}Say $j_1$, \ldots, $j_n$ are nonnegative integers
and $a_i\in A_{\leq j_i}$ are such that the classes $a_i+A_{< j_i}\in \gr_{j_i}A$ generate $\gr A$. Then the $a_i$
generate $A$.
\end{proof}
\begin{Lem}\label{U inv fg fg:Lem}
Suppose $\k$ is Noetherian. If $A^U$ is a finitely generated $\k$-algebra, so is $ A$.
\end{Lem}
\begin{proof}
Combine Lemmas \ref{fg hull:Lem}, \ref{hull fg gr fg:Lem}, \ref{gr fg fg:Lem}.
\end{proof}
\begin{Lem}\label{Ru inv fg fg:Lem}
Let $P$ be a standard parabolic subgroup. Suppose $\k$ is Noetherian. Then $A$ is a finitely generated $\k$-algebra if and only if $A^{R_u(P)}$ is one.
\end{Lem}
\begin{proof}
Let $V$ be the intersection of $U$ with the semisimple part of the standard Levi subgroup of $P$.
Then $U=VR_u(P)$ and $A^U=(A^{R_u(P)})^V$.
Suppose that $A$ is a finitely generated $\k$-algebra. 
Then $A^U=(A^{R_u(P)})^V$ is one also by Lemma \ref{U inv fg:Lem}, 
and so is $A^{R_u(P)}$ by Lemma \ref{U inv fg fg:Lem} (applied with a different group and a different algebra). 
\par
Conversely, if $A^{R_u(P)}$ is a finitely generated $\k$-algebra, Lemma \ref{U inv fg:Lem}
(with that same group and algebra) implies that $A^U=(A^{R_u(P)})^V$ is finitely generated,
and thus $A$ is as well, by Lemma \ref{U inv fg fg:Lem}.
\end{proof}
\begin{proof}[Proof of Theorem \ref{hull and finite generation:theorem}]
Combine Lemmas \ref{Ru inv fg fg:Lem}, \ref{U inv fg:Lem}, \ref{fg hull:Lem}, \ref{hull fg gr fg:Lem}, \ref{gr fg fg:Lem}.
\end{proof}
\begin{proof}[Proof of Theorem \ref{torsion:theorem}]
Let $\k$ be Noetherian and let $A$ be a finitely generated $\k$-algebra. 
By Theorem \ref{hull and finite generation:theorem},
the $\k$-algebra $\hull_\nabla(\gr A)$ is finitely generated. 
So we may choose $b_1$,\ldots,$b_s$, so that 
$\psi_{b_1}\ot\cdots\ot\psi_{b_s}$ has image $\hull_\nabla(\gr A)$. 
By extending Lemma \ref{torsion:Lem} to tensor products,
we can argue as in the proof of Lemma \ref{powers weight:Lem} and Theorem \ref{hull and powers:theorem},
and see that there is a positive integer $n$ so that $n\hull_\nabla(\gr A)\subseteq \gr A$. 
Now,
$\hull_\nabla(\gr A)\ot \k[G/U]$ is acyclic by Proposition \ref{coh crit:Pro}, and thus its summand
$\hull_\nabla(\gr A)$ is acyclic as well. 
It follows that $\h^i(G,\gr A)$ is a quotient of $\h^{i-1}(G,\hull_\nabla(\gr A)/\gr A)$,
which is annihilated by $n$.
\end{proof}
\begin{proof}[Proof of Theorem \ref{generic g G:theorem}]
Take $n$ as in Theorem \ref{torsion:theorem}, and use that localization is exact.
\end{proof}

\begin{proof}[Proof of Theorem \ref{gr power:theorem}]
It suffices to show that the composite:
$$\
\gr A\to \gr (A/pA)\to \hull_\nabla(\gr (A/pA))
$$
is power-surjective. It coincides with the composite
$$\
\gr A\to \hull_\nabla(\gr (A))\to \hull_\nabla(\gr (A/pA)).
$$
Now $A^{U^+}\to (A/pA)^{U^+}$ is $p$-power-surjective by a combination of 
Theorem \ref{Mum:theorem}, Proposition \ref{power:Pro}, Proposition \ref{co t:Prop}
and the transfer principle \cite[Ch. Two]{G} as used in \ref{U inv fg:Lem}.
After inducing up, $\hull_\nabla(\gr (A)) \to \hull_\nabla(\gr (A/pA))$ is still $p$-power-surjective,
indeed the same $p$-power is sufficient. 
And $\gr A\to \hull_\nabla(\gr (A))$ is power-surjective by Theorem \ref{hull and powers:theorem}.
\end{proof}
\begin{commentVF}
I got a little lost for this proof, please check and feel free to revert!
\end{commentVF}
\begin{comment}
Looks fine to me. It uses base change for costandard modules effectively.
\end{comment}
\section{Finiteness properties of cohomology algebras}\label{finiteness:section}

In this section we study finiteness properties of
$\h^*(G,A)$, primarily when the base ring $\k$ is $\Z$. 
We shall always assume that the commutative algebra $A$ is finitely generated over the ring $\k$, with rational action of a Chevalley group scheme $G$.
Further, $M$ will be a noetherian $A$-module with compatible $G$-action.
Torsion will refer to torsion as an abelian group, not as an
$A$-module. We say that $V$ has bounded torsion if there is a positive integer that annihilates
the torsion subgroup $V_\tors$. 
\begin{comment}
So then $V_\tors$ has bounded exponent.
\end{comment}

\begin{Lem}A noetherian module over a
graded commutative ring has bounded torsion.\qed
\end{Lem}
Recall that we call a homomorphism $f$ : $R\to S$ of graded commutative algebras \emph{noetherian} if $f$ makes $S$ into a noetherian
left $R$-module. Recall that CFG refers to cohomological finite generation.
The main result of this section is the following.
\begin{theorem}[Provisional CFG]\label{bounded power fg:theorem}
Suppose $\k=\Z$.
\begin{itemize}
\item Every $\h^m(G,M)$ is a noetherian $A^G$-module.
\item If $\h^*(G,A)$ is a finitely generated algebra, then $\h^*(G,M)$ is a noetherian $\h^*(G,A)$-module.
\item $\h^*(G,\gr A)$ is a finitely generated algebra.
\item If\/ $\h^*(G,A)$ has bounded torsion, then the reduction $\h^\even(G,A)\to \h^\even(G,A/pA)$
is power-surjective for every prime number $p$.
\item If $\h^\even(G,A)\to \h^\even(G,A/pA)$
is noetherian for every prime number $p$, then $\h^*(G, A)$ is a finitely generated algebra.
\end{itemize}
\end{theorem}
\begin{Rem}
Note that the first statement would fail badly, by \cite[I 4.12]{J}, if one replaced $G$ with the additive group scheme $\G_a$.
This may explain why our proof is far from elementary.
\end{Rem}
We hope to show in the future that $\h^\even(G,A)\to \h^\even(G,A/pA)$
is noetherian for every prime number $p$. 
The theorem suggests to ask:
\begin{Pb}\label{power surj conj:Pb}
Let $\k$ be a Noetherian ring and let $G$ be a Chevalley group scheme over $\k$.
Let $A$, $Q$ be finitely generated commutative $\k$-algebras on which $G$ acts rationally through algebra automorphisms. 
Let $f:A\to Q$ be a power-surjective equivariant homomorphism.
Is $\h^*(G,A)\to \h^*(G,Q)$ power-surjective?
\end{Pb}
We will need the recent theorem of Touz\'e \cite[Thm 1.1]{TvdK}, see also \cite[Thm 1.5]{TvdK},
\begin{theorem}[CFG over a field]\label{(CFG)}If $\k$ is a field, then
$\h^*(G,A)$ is a finitely generated $\k$-algebra and $\h^*(G,M)$ is
a noetherian $\h^*(G,A)$-module.
\end{theorem}
\begin{Rem}
If $\k$ is a commutative ring and $V$ is a $G_\k$-module, then the comultiplication $V\to V\otimes_\k \k[G]$
gives rise to a comultiplication $V\to V\otimes_\Z \Z[G]$ through the identification $V\otimes_\k \k[G]=V\otimes_\Z \Z[G]$.
So one may view $V$ as a $G_\Z$-module.
Further
$\h^*(G_\k,V)$ is the same as $\h^*(G_\Z,V)$, because the Hochschild complexes are the same.
So if $\k$ is finitely generated over a field $F$, then the conclusions of the (CFG) theorem still hold,
because $\h^*(G,A)=\h^*(G_F,A)$. 
We leave it to the reader to try a limit argument to deal with the case where $\k$ is
essentially of finite type over a field. 
\end{Rem}
%
\begin{comment}
We now aim for 
\begin{theorem}
The following are equivalent
\begin{itemize}
\item $\h^*(G,A)$ has bounded torsion
\item $\h^\even(G,A)\to \h^\even(G,A/pA)$ is power-surjective for every prime number $p$
\item $\h^*(G,A)$ is a finitely generated $\Z$ algebra
\end{itemize}
\end{theorem}
%
\begin{Cor}
$\h^*(G,\gr A)$ is a finitely generated $\Z$ algebra
\end{Cor}
%
\begin{Cor}
$\h^m(G,A)$ is a noetherian $A^G$ module for every $m$
\end{Cor}
\end{comment}

First let the ring $\k$ be noetherian.
We are going to imitate arguments of Benson--Habegger \cite{BH}. We thank Dave Benson for the reference.

\begin{Lem}\label{power:Lem}
Let $m>1$, $n>1$. The reduction $\h^\even(G,A/mnA)\to \h^\even(G,A/nA)$ is power-surjective.
\end{Lem}
\begin{proof}We may assume $m$ is prime. By the Chinese Remainder Theorem we may then also
assume that $n$ is a power of that same prime. (If $n$ is prime to $m$ the Lemma is clear.)
Let $x\in \h^\even(G,A/nA)$. We show that some power $x^{m^r}$ of $x$
lifts. Arguing as in the proof of Proposition \ref{co t:Prop}
we may assume $x$ is homogeneous. Let $I$ be the kernel of $A/mnA\to A/nA$. Note that $m$ annihilates $I$, hence also
$\h^*(G,I)$.
Further $I$ is an $A/nA$-module and the connecting homomorphism
$\partial:\h^i(G,A/nA)\to \h^{i+1}(G,I)$
satisfies the Leibniz rule. So $\partial(x^m)=mx^{m-1}\partial(x)=0$ and $x^m$ lifts.
\end{proof}

\begin{Pro}
If\/ $\h^*(G,A)$ has bounded torsion, then $\h^\even(G,A)\to \h^\even(G,A/pA)$
is power-surjective for every prime number $p$.
\end{Pro}
\begin{proof}
Assume $\h^*(G,A)$ has bounded torsion.
Write $\h^\pos$ for $\bigoplus_{i>0}\h^i$.
Let $p$ be a prime number. 
Choose a positive multiple $n$ of $p$ so that $n\h^\pos(G,A)=0$ and $nA_\tors=0$.
We have an exact sequence 
$$\cdots\to \h^i(G,A_\tors)\to \h^i(G,A)\to \h^i(G,A/A_\tors)\to\cdots.$$
Multiplication by $n^2$ is zero on $\h^\pos (G,A/A_\tors)$, so $\h^\pos (G,A/A_\tors)\hookrightarrow \h^\pos(G,A/n^2A+A_\tors)$.
We have exact sequences
$$0\to A/A_\tors\stackrel{\times n^2}\longrightarrow A\to A/n^2A\to 0$$ and
$$0\to A/n^2A+A_\tors\stackrel{\times n^2}\longrightarrow A/n^4A\to A/n^2A\to 0.$$
Consider the diagram
\[\xymatrix{
\ar[r]\ar[d]\h^{2i}(G,A)&\ar[r]^{\partial_1}\ar@{=}[d] \h^{2i}(G,A/n^2A)&\ar@{^{(}->}[d] \h^{2i+1}(G,A/A_\tors) \\
\ar[r]\h^{2i}(G,A/n^4A)&\ar[r]^{\partial_2}\h^{2i}(G,A/n^2A) &\h^{2i+1}(G,A/n^2A\rlap{$+A_\tors)$}
}\]
If $x\in \h^{2j}(G,A/n^2A)$, put $i=jn^2$. The image $n^2x^{n^2-1}\partial_2(x)$ 
in $\h^{2i+1}(G,A/n^2A+A_\tors)$ of $x^{n^2}$ vanishes, hence $\partial_1(x^{n^2})$ vanishes in $\h^{2i+1}(G,A/A_\tors)$,
and $x^{n^2}$ lifts to $\h^{2i}(G,A)$. 
As $\h^\even(G,A/n^2A)\to \h^\even(G,A/pA)$ is power-surjective by Lemma \ref{power:Lem}, we conclude that for every
homogeneous $y\in \h^\even(G,A/pA)$ some power lifts all the way to $\h^\even (G,A)$.
We want to show more, namely that $\h^\even(G,A)\to \h^\even(G,A/pA)$ is universally
power-surjective.
By Proposition \ref{co t:Prop} we need to show that the power of $y$ may be taken of the form
$y^{p^r}$. Localize with respect to the multiplicative system $S=(1+p\Z)$
in $\Z$. The $p$-primary torsion is not affected and all the other torsion disappears, 
so $n$ may be taken a power of $p$. The proofs then produce that some
$y^{p^r}$ lifts to $\h^\even (G,S^{-1}A)$. Now just remove the denominator, which acts trivially on $y$.
\end{proof}
We now restrict to the case $\k=\Z$. 
(More generally, one could take for $\k$ a noetherian ring so that for every prime
number $p$ the ring $\k/p\k$ is essentially of finite type over a field.) 
\begin{Pro}\label{bounded fg:Pro}
Suppose $\k=\Z$.
If $\h^*(G,A)$ has bounded torsion, then $\h^*(G, A)$ is a finitely generated algebra.
\end{Pro}
\begin{proof}
By Theorem \ref{generic g G:theorem} we may choose a prime number $p$ and concentrate on the $p$-primary part.
Say by tensoring $\Z$ and $A$ with $\Z_{(p)}$. So now $\h^\pos(G,A)$ is $p$-primary torsion and $A$ is a $\Z_{(p)}$-algebra.
We know that $\h^\even(G,A)\to \h^\even(G,A/pA)$ is power-surjective.
By power surjectivity and
the (CFG) Theorem \ref{(CFG)}, we choose an $A^G$-subalgebra $C$ 
of $\h^*(G,A)$, generated by finitely many
homogeneous elements, so that $C\to \h^*(G,A/pA)$ is noetherian. 
Again by the (CFG) Theorem \ref{(CFG)} it follows that
$\h^*(G,A/pA)\to \h^*(G,A/pA+A_\tors)$ is noetherian, so that
$C\to \h^*(G,A/pA+A_\tors)$ is also noetherian.
\par
Let $N$ be the image of $\h^\pos(G,A/A_\tors)$ in $\h^\pos(G,A/pA+A_\tors))$. 
As a $C$-module, it is isomorphic to $\h^\pos(G,A/A_\tors)/p\h^\pos(G,A/A_\tors)$. 
Choose homogeneous $v_i\in \h^\pos(G,A/A_\tors)$ so that their images generate $N$.
Say $V$ is the $C$-span of the $v_i$.
We have $\h^\pos(G,A/A_\tors)+V\subseteq p\h^\pos(G,A/A_\tors)+V$. Iterating this we get
$\h^\pos(G,A/A_\tors)+V\subseteq p^r\h^\pos(G,A/A_\tors)+V$ for any $r>0$.
But $\h^*(G,A)$ and $\h^*(G,A_\tors)$
have bounded torsion, so $\h^*(G,A/A_\tors)$ also has bounded torsion. It follows that
$\h^\pos(G,A/A_\tors)=V$. We conclude that $\h^\pos(G,A/A_\tors)$ is a noetherian $C$-module.
\par
Now let us look at $\h^\pos(G,A_\tors)$. Filter $A_\tors\supseteq pA_\tors\supseteq p^2A_\tors\supseteq\cdots \supseteq0$.
By the (CFG) theorem $\h^\pos(G,p^iA_\tors/p^{i+1}A_\tors)$ is a noetherian $\h^*(G,A/pA)$-module, hence a noetherian $C$-module.
So $\h^\pos(G,A_\tors)$ is also a noetherian $C$-module and thus $\h^*(G,A)$ is one. It follows that $\h^*(G,A)$ is a finitely generated
$A^G$-algebra. And $A^G$ itself is finitely generated by Theorem \ref{FFT:theorem}.
\end{proof}
\begin{Pro}Let $\k=\Z$.
Then $\h^*(G,\gr A)$ is a finitely generated algebra.
\end{Pro}
\begin{proof}
By Theorem \ref{torsion:theorem} the algebra $\h^*(G,\gr A)$ has bounded torsion, so Proposition \ref{bounded fg:Pro} applies.
\end{proof}
\begin{Pro}\label{Hm is noetherian:Pro}
Let $\k=\Z$.
Then $\h^m(G,M)$ is a noetherian $A^G$-module.
\end{Pro}
\begin{proof}
Form the `semi-direct product ring' $A\ltimes M$ whose underlying $G$-module
is $A\oplus M$, with product given by
$(a_1,m_1)(a_2,m_2)=(a_1a_2,a_1m_2+a_2m_1)$.
It suffices to show that $\h^m(G,A\ltimes M)$ is a noetherian $\h^0(G,A\ltimes M)$-module.
In other words, we may forget $M$ and just ask if $\h^m(G,A)$ is a noetherian $A^G$-module.
Now $\h^*(G,\gr A)$ is a finitely generated algebra and $\h^0(G,\gr A)=\gr_0 A$ , so in the spectral sequence 
$$E(A):E_1^{ij}=\h^{i+j}(G,\gr_{-i}A)\Rightarrow \h^{i+j}(G, A)$$
the $\bigoplus_{i+j=t}E_1^{ij}$ are noetherian $A^G$-modules for each $t$.
So for fixed $t$ there are only finitely many nonzero $E_1^{i,t-i}$ and the result follows.
\end{proof}
\begin{Pro}\label{noeth cfg:Pro}Let $\k=\Z$.
If $\h^\even(G,A)\to \h^\even(G,A/pA)$
is noetherian for every prime number $p$, then $\h^*(G, A)$ is a finitely generated algebra.
\end{Pro}
\begin{proof}
We argue as in the proof of Proposition \ref{bounded fg:Pro}. 
We may no longer know that $\h^*(G, A)$ has bounded torsion, but for every $m>0$ we know that $\h^m(G,A/A_\tors)$ is a noetherian 
$A^G$-module, hence has bounded torsion. Instead of
$\h^\pos(G,A/A_\tors)+V\subseteq p\h^\pos(G,A/A_\tors)+V$, we use $\h^m(G,A/A_\tors)+V\subseteq p\h^m(G,A/A_\tors)+V$.
We find that $\h^m(G,A/A_\tors)\subseteq V$ for all $m>0$ and thus $\h^\pos(G,A/A_\tors)= V$ again.
Finish as before.
\end{proof}
\begin{Cor}\label{pow noeth cfg:Pro}Let $\k=\Z$.
If $\h^\even(G,A)\to \h^\even(G,A/pA)$
is power-surjective for every prime number $p$, then $\h^*(G, A)$ is a finitely generated algebra.
\end{Cor}
\begin{Pro}Let $\k=\Z$.
If $\h^*(G,A)$ is a finitely generated algebra, then $\h^*(G,M)$ is a noetherian $\h^*(G,A)$-module.
\end{Pro}
\begin{proof}
Let $\h^*(G,A)$ be a finitely generated algebra. So it has bounded torsion
and $\h^\even(G,A)\to \h^\even(G,A/pA)$ is power-surjective for every prime number $p$.
We argue again as in the proof of Proposition \ref{bounded fg:Pro}.
\par
By Theorem \ref{generic g G:theorem}, applied to $A\ltimes M$, we may choose a prime number $p$ and concentrate on the $p$-primary part, so $\h^\pos(G,M)$ is $p$-primary torsion and $A$ is a $\Z_{(p)}$-algebra.
Write $C=\h^*(G,A)$. 
By power surjectivity and the (CFG) Theorem \ref{(CFG)}, $C\to \h^*(G,A/pA)$ is noetherian. Again by the (CFG) Theorem \ref{(CFG)} it follows that
$\h^*(G,M/pM+M_\tors)$ is a noetherian $\h^*(G,A/pA)$-module, hence a noetherian $C$-module.
\par
Let $N$ be the image of $\h^\pos(G,M/M_\tors)$ in $\h^\pos(G,M/pM+M_\tors))$. 
As a $C$-module, it is isomorphic to $\h^\pos(G,M/M_\tors)/p\h^\pos(G,M/M_\tors)$. 
Choose homogeneous $v_i\in \h^\pos(G,M/M_\tors)$ so that their images generate $N$. 
Say $V$ is the $C$-span of the $v_i$.
We have $\h^m(G,M/M_\tors)+V\subseteq p\h^m(G,M/M_\tors)+V$ for $m>0$. Iterating this we get
$\h^m(G,M/M_\tors)+V\subseteq p^r\h^m(G,M/M_\tors)+V$ for any $r>0$, $m>0$.
But $\h^m(G,M/M_\tors)$ is a noetherian $A^G$-module, hence has
bounded torsion. It follows that $\h^m(G,M/M_\tors)\subseteq V$ for all $m>0$, and
$\h^\pos(G,M/M_\tors)=V$. So $\h^\pos(G,M/M_\tors)$ is a noetherian $C$-module.
\par 
Now let us look at $\h^\pos(G,M_\tors)$. Filter $M_\tors\supseteq pM_\tors\supseteq p^2M_\tors\supseteq\cdots \supseteq0$.
By the (CFG) theorem $\h^\pos(G,p^iM_\tors/p^{i+1}M_\tors)$ is a noetherian $\h^*(G,A/pA)$-module, hence a noetherian $C$-module.
So $\h^\pos(G,M_\tors)$ is also a noetherian $C$-module and thus $\h^*(G,M)$ is one. 
\end{proof}

\begin{comment}
In particular, to prove (CFG) it suffices to consider the case of a symmetric algebra $S^*(M)$
for some finitely generated representation $M$.
If you wish, you could use Bob Thomason's result to have the representation projective as $\k$-module. 
\end{comment}
Theorem \ref{bounded power fg:theorem} has been proven.



\Addresses
\end{document}